\documentclass{amsart}
\usepackage{amsmath}
\usepackage{amssymb}
\usepackage{mathabx}
\usepackage[all]{xy}
\usepackage{amsbsy}
\usepackage{graphicx}
\usepackage{appendix}
\usepackage{float}
\usepackage{subfigure}
\usepackage[numbers,sort&compress]{natbib}
\usepackage{graphicx}
\usepackage{amsthm}
\usepackage{geometry}
\usepackage{fancyhdr}
\usepackage{color} 
\usepackage{overpic}
\usepackage{mathabx}
\usepackage{tikz-cd}
\usepackage{enumerate}
\usepackage{mathrsfs}
\usepackage{colonequals}
\usepackage{microtype}

\newtheorem{thm}{Theorem}[section]
\newtheorem{cor}[thm]{Corollary}
\newtheorem{prop}[thm]{Proposition}
\newtheorem{lem}[thm]{Lemma}

\theoremstyle{definition}
\newtheorem{defn}[thm]{Definition}
\newtheorem{exmp}[thm]{Example}

\theoremstyle{remark}
\newtheorem{rem}[thm]{Remark}

\numberwithin{equation}{section}

\newcommand{\beq}{\begin{equation*}\begin{aligned}}
\newcommand{\eeq}{\end{aligned}\end{equation*}}
\newcommand{\bpf}{\begin{proof}}
\newcommand{\epf}{\end{proof}}
\newcommand{\bthm}{\begin{thm}}
\newcommand{\ethm}{\end{thm}}
\newcommand{\bprop}{\begin{prop}}
\newcommand{\eprop}{\end{prop}}
\newcommand{\bcor}{\begin{cor}}
\newcommand{\ecor}{\end{cor}}
\newcommand{\blem}{\begin{lem}}
\newcommand{\elem}{\end{lem}}
\newcommand{\bdefn}{\begin{defn}}
\newcommand{\edefn}{\end{defn}}
\newcommand{\bexmp}{\begin{exmp}}
\newcommand{\eexmp}{\end{exmp}}
\newcommand{\brem}{\begin{rem}}
\newcommand{\erem}{\end{rem}}
\newcommand{\benu}{\begin{enumerate}[(1)]}
\newcommand{\eenu}{\end{enumerate}}

\newcommand{\bdia}{\begin{displaymath}\xymatrix}
\newcommand{\edia}{\end{displaymath}}

\newcommand{\al}{\alpha}
\newcommand{\be}{\beta}
\newcommand{\ga}{\gamma}

\newcommand{\intg}{\mathbb{Z}}

\newcommand{\real}{\mathbb{R}}

\newcommand{\ra}{\rightarrow}

\begin{document}

\title{Some computations on instanton knot homology}


\author{Zhenkun Li}
\address{Department of Mathematics, Stanford University}
\curraddr{}
\email{zhenkun@stanford.edu}
\thanks{}

\author{Yi Liang}
\address{Massachusetts Institute of Technology}
\curraddr{}
\email{liangy@mit.edu}
\thanks{}

\date{}
\dedicatory{}
\maketitle
\begin{abstract}
In a recent paper, the first author and his collaborator developed a method to compute an upper bound of the dimension of instanton Floer homology via Heegaard diagrams of $3$-manifolds. In this paper, for a knot inside $S^3$, we further introduce an algorithm that computes an upper bound of the dimension of instanton knot homology from knot diagrams. We test the algorithm with all knots up to $7$ crossings as well as a more complicated knot $10_{153}$. In the second half of the paper, we show that if the instanton knot Floer homology of a knot has a specified form, then the knot must be an instanton L-space knot.
\end{abstract}


\section{Introduction}
Knot theory is a central topic in low dimensional topology. In 1990, Floer introduced a knot invariant called the instanton knot homology in \cite{floer1990knot}. It has become a powerful tool in the study of knot theory. For example, the instanton knot homology detects the genus and fibredness of a knot, recovers the Alexander polynomial, and plays an important role in the establishment of the milestone result that Khovanov homology detects the unknot. See Kronheimer and Mrowka \cite{kronheimer2010knots,kronheimer2010instanton,kronheimer2011knot,kronheimer2011khovanov}.

The instanton knot homology of a knot $K\subset S^3$ is a finite-dimensional complex vector space associated to $K$. It is constructed by studying the solutions of sets of partial differential equations over a closed oriented $3$-manifold $Y$ and the infinite cylinder $\real\times Y$, where $Y$ is obtained from the knot complement by attaching some standard piece of $3$-manifold that also has a toroidal boundary. This nature of instanton knot homology makes it very difficult to compute. The breakthrough in computation was made by Kronheimer and Mrowka. Combining their work that the Euler characteristic of instanton knot homology recovers the Alexander polynomial of the knot, and that there exists a spectral sequence whose $E_2$ page is the Khovanov homology and whose $E_{\infty}$ page is the instanton knot homology. Thus one obtains a lower bound of the dimension of the instanton knot homology via the Alexander polynomial and an upper bound via the Khovanov homology. When these two bounds coincide, for example, for all alternating knots, the computation is done. 

Later, several groups of people studied the representation varieties of some special families of knots to write down an explicit set of generators of the chain complex of the instanton knot homology and thus obtained some better upper bounds than the one coming from Khovanov homology. See \cite{Hedden2014,daemi2019equivariant,Lobb2020}.

Heegaard diagram is an effective combinatorial way to describe $3$-manifolds and knots. In fact, any knot can be described and is determined by its Heegaard diagrams. Since instanton knot homology serves as a knot invariant, it is a priori determined by the Heegaard diagram of the knot. However, its construction through partial differential equations makes its relation to Heegaard diagrams very implicit. To study this relation, recently, the first author and his collaborator established the following result in \cite{li2020heegaard}. 
\bthm\label{thm: key inequality}
Suppose $K\subset S^3$ is a knot, and $(\Sigma,\al,\be)$ is a Heegaard diagram of $K$. Let $H$ be a handle body with $\partial H=\Sigma$ and $\ga\subset H$ is a (disconnected) oriented simple closed curve on $H$ so that the following is true.
\begin{enumerate}
\item We have that $\ga$ has $(g(\Sigma)+1)$ components.
\item We have that $\Sigma\backslash\ga$ consists of two components of equal Euler characteristics.
\item We have that all $\be$-curves are components of $\ga$.
\end{enumerate}
Then we have the following inequality.
$$\dim_{\mathbb{C}}KHI(K)\leq\dim_{\mathbb{C}}SHI(H,\ga).$$
Here, $KHI$ is the instanton knot homology of $K$, and $SHI$ is the sutured instanton Floer homology of the balanced sutured manifold $(H,\ga)$.
\ethm

Sutured instanton Floer homology associates a finite-dimensional complex vector space to every balanced sutured manifold. It was introduced by Kronheimer and Mrowka in \cite{kronheimer2010knots}. The computation of $SHI(M,\ga)$ for a general balanced sutured manifold $(M,\ga)$ is also difficult, due to the same reason as $KHI$. However, in the special case when $M=H$ is a handle body, the first author and his collaborator developed an algorithm to compute an upper bound of $SHI(H,\ga)$ in \cite{li2019decomposition}. So, equipped with Theorem \ref{thm: key inequality}, one can obtain an upper bound on the dimension of the instanton knot homology of a knot from any Heegaard diagram of that knot. Following this idea, Li and Ye were able to compute a conjecturally sharp upper bound for all $(1,1)$-knots in \cite{li2020heegaard}.

In this paper, we utilize Theorem \ref{thm: key inequality} further and obtain the following.

\bthm\label{thm: algorithm}
Suppose $K\subset S^3$ is a knot. Let $D$ be any knot diagram of $K$. Then there is an algorithm to compute an upper bound of the dimension of $KHI(S^3,K)$ out of $D$. 
\ethm

\brem
The project in the current paper was launched right after the completion of \cite{li2020heegaard}, where the first author and his collaborator first introduced Theorem \ref{thm: key inequality}. Later, Theorem \ref{thm: key inequality} was further used to prove the main result of \cite{BLY2020}. By \cite[Theorem 1.1]{BLY2020}, one knows that the dimension of $KHI(S^3,K)$ is bounded by the number of generators of $\widehat{CFK}(S^3,K)$, where $\widehat{CFK}$ is the chain complex of the hat version of knot Floer homology introduced by Ozsv\'ath and Szab\'o \cite{ozsvath2004holomorphicknot}. Since the number of generators of $\widehat{CFK}(S^3,K)$ can also be computed directly from any given diagram of the knot $K$, \cite[Theorem 1.1]{BLY2020} gives rise to an algorithm that is different from the one in Theorem \ref{thm: algorithm}. Though we haven't compared the effectiveness between these two algorithms.
\erem

We further testify the algorithm in Theroem \ref{thm: algorithm} with knots of crossing numbers at most $7$. We found that all the bounds from the algorithm are sharp. We also test a more complicated knot $10_{153}$. See Table \ref{tab: knots} for more details.

It is worth mentioning that all knots of crossing number at most $7$ are alternating. So the dimensions of $KHI$ for them have been known due to the work of Kronheimer and Mrowka \cite{kronheimer2010instanton,kronheimer2011khovanov} as discussed above. However, the upper bounds coming from Khovanov homology depends on the establishment of a spectral sequence that relates instanton knot homology with Khovanonv homology. So Theorem \ref{thm: algorithm} provides an alternative proof that is independent of Kronheimer and Mrowka's spectral sequence.

Instanton L-space knots are those knots inside $S^3$ that admits a Dehn surgery whose instanton Floer homology has minimal dimension. As mentioned above, the computation of instanton Floer homology is a big open problem, so in general, it is hard to identify instanton L-space knots. In this paper, we give a sufficient condition in terms of the instanton knot homology.

\begin{thm}\label{thm : main theorem}
Suppose $K\subset S^3$ is a knot of genus $g$. Suppose further that
\begin{equation}\label{eq: intro}
    KHI(K, i)\cong \left\{
                \begin{array}{ll}
                  \mathbb{C}, & \mid i \mid = g, g-1, 0\\ 
                   0, & Otherwise\\
                \end{array}
              \right.
\end{equation}
Then $K$ admits an instanton L-space surgery.
\end{thm}

Combined with the results from \cite{lidman2020framed,baldwin2020concordance}, we conclude the following.
\bcor\label{cor: intro}
Suppose $K\subset S^3$ is a knot of genus $g$ whose instanton knot homology is described as in (\ref{eq: intro}), then one and exactly one of the following two statements is true.

(1). $S_{2g-1}(K)$ is an instanton L-space, and for any rational number $r=\frac{p}{q}$  with $q\geq 1$, we have
\begin{equation}
{\rm dim}_{\mathbb{C}}I^{\sharp}(S^3_r(K))=\left\{
\begin{array}{cc}
    p & {\rm if~}r\geq 2g-1 \\
    (4g-2)\cdot q-p & {\rm otherwise}
\end{array}
\right.
\end{equation}

(2). $S_{1-2g}(K)$ is an instanton L-space, and for any rational number $r=\frac{p}{q}$ with $q\geq 1$, we have
\begin{equation}
{\rm dim}_{\mathbb{C}}I^{\sharp}(S^3_r(K))=\left\{
\begin{array}{cc}
    -p & {\rm if~}r\leq 1-2g \\
    (4g-2)\cdot q+p & {\rm otherwise}
\end{array}
\right.
\end{equation}
\ecor

\section{Preliminaries}\label{sec: preliminaries}
\bdefn[\cite{juhasz2006holomorphic,kronheimer2010knots}]\label{defn_2: balanced sutured manifold}
A \textbf{balanced sutured manifold} $(M,\ga)$ consists of a compact oriented 3-manifold $M$ with non-empty boundary together with a closed 1-submanifold $\ga$ on $\partial{M}$. Let $A(\ga)=[-1,1]\times\ga$ be an annular neighborhood of $\ga\subset \partial{M}$ and let $R(\ga)=\partial{M}\backslash{\rm int}(A(\ga))$. They satisfy the following properties. 
\begin{enumerate}
    \item Neither $M$ nor $R(\ga)$ has a closed component.
    \item If $\partial{A(\ga)}=-\partial{R(\ga)}$ is oriented in the same way as $\ga$, then we require this orientation of $\partial{R(\ga)}$ induces one on $R(\ga)$. The induced orientation on ${R(\ga)}$ is called the \textbf{canonical orientation}.
    \item Let $R_+(\ga)$ be the part of $R(\ga)$ so that the canonical orientation coincides with the induced orientation on $\partial{M}$, and let $R_-(\ga)=R(\ga)\backslash R_+(\ga)$. We require that $\chi(R_+(\ga))=\chi(R_-(\ga))$. If $\ga$ is clear in the contents, we simply write $R_\pm=R_\pm(\ga)$, respectively.
\end{enumerate}
\edefn

\bthm[Kronheimer and Mrowka \cite{kronheimer2010knots}]
For any balanced sutured manifold $(M,\ga)$, we can associate a finite-dimensional complex vector space, which we denote by $SHI(M,\ga)$, to $(M,\ga)$. It serves as a topological invariant of the pair $(M,\ga)$.
\ethm

\bdefn[Kronheimer and Mrowka \cite{kronheimer2010knots}]
For knot $K\subset S^3$, define its {\bf instanton knot homology}, which we denote by $KHI(K)$, to be
$$KHI(K)=SHI(S^3(K),\ga_{\mu}),$$
where $\ga_{\mu}$ consists of two meridians of $K$.

For a connected closed oriented $3$-manifold $Y$, define its {\bf framed instanton Floer homology}, which is denoted by $I^{\sharp}(Y)$, to be
$$I^{\sharp}(Y)=SHI(Y(1),\delta),$$
where $Y(1)=Y\backslash D^3$ is obtained from $Y$ by removing a $3$-ball and $\delta\subset\partial Y(1)$ is a connected simple closed curve.
\edefn

\bdefn
A knot $K\subset S^3$ is called an {\bf instanton L-space knot} if there exists a non-zero integer $n$ so that
$${\rm dim}_{\mathbb{C}}I^{\sharp}(S^3_n(K))=|n|.$$
Here $S^3_n(K)$ is the three manifold obtained from $S^3$ by an $n$ surgery along the knot $K$.
\edefn

In \cite{kronheimer2010knots,kronheimer2010instanton}, Kronheimer and Mrowka studied many basic properties of $KHI$, which we summarize as in the following two theorems.

\begin{thm}\label{thm: properties of KHI}
Suppose $K\subset S^3$ is a knot of genus $g$, then the following is true.  
\begin{enumerate}
	\item There is a $\intg$-grading on $KHI(K)$, which is called the {\bf Alexander grading}:
$$KHI(K) = \bigoplus _{i \in \mathbb{Z}} KHI(K, i).$$
\item For any $i \in \intg$ with $\mid i \mid >g$, we have $KHI(K,i)=0$.
\item We have $KHI(K,g)\neq 0$.
\item For any $i\in\intg$, we have $KHI(K, i) \cong KHI(k, -i)$.
\item The knot $K$ is fibred if and only if $KHI(K,g)\cong\mathbb{C}$.
\end{enumerate}
\end{thm}

\bthm[Kronheimer and Mrowka \cite{kronheimer2010instanton}]
Suppose $K\subset S^3$ is a knot. Let
$$\Delta_K(t)=\sum_{i\in\intg}a_it^i$$
be its Alexander polynomial. Then we know that
$$\dim_{\mathbb{C}}KHI(K)\geq \sum_{i\in\intg}|a_i|.$$
Here $|\cdot|$ means the abstract value.
\ethm

In \cite{li2019direct,li2019tau,li2020heegaard}, the first author and his collaborators studied different sutures on the knot complements. Suppose $K\subset S^3$ is a knot. Let $\ga_{(p,q)}$ be the suture on $\partial S^3(K)$ consisting of two simple closed curves of slope $q/p$ on $\partial S^3(K)$. We have the following.

\begin{thm}\label{thm: properties of SHI}
Suppose $K\subset S^3$ is a knot of genus $g$. For any pair of co-prime integers $(p,q)\in\intg^2$, $SHI(S^3(K), \gamma_{(p,q)})$ admits a grading that sits in either $\intg$ or $\intg+\frac{1}{2}$:

If q is odd, then

 \[SHI(S^3(K), \gamma_{(p,q)}) = \bigoplus_{i \in \mathbb{Z}} SHI(S^3(K), \gamma_{(p,q)}, i).\]
 
 If q is even, then
 
 \[SHI(S^3(K), \gamma_{(p,q)}) = \bigoplus_{i \in \mathbb{Z} + \frac{1}{2}} SHI(S^3(K), \gamma_{(p,q)}, i).\]

Furthermore, the following is true.
\begin{enumerate}
\item For any $i$ with $\mid i \mid > g + \frac{q-1}{2}$, we have 
$$SHI(S^3(K), \gamma_{(p,q)}, i)= 0$$
\item We have 
$$SHI(S^3(K), \gamma_{(p,q)}, g + \frac{q-1}{2}) \not = 0.$$
\item For any $i$, we have
$SHI(S^3(K), \gamma_{(p,q)}, i) \cong SHI(S^3(K), \gamma_{(p,q)}, -i)$.
\item We have
\begin{equation}\label{eq: middle grading}
	SHI(S^3(K),\ga_{(1,-2g-1)},0)\cong\mathbb{C}.
\end{equation}
\item We have
\begin{equation}\label{eq: I-sharp and SHI}
	I^{\sharp}(S^{3}_{-2g-1}(K))\cong\bigoplus_{i=-g}^gSHI(S^3(K),\ga_{(2,-4g-1)},i).
\end{equation}

\end{enumerate}
\end{thm}

Bypass triangles were introduced in the instanton theory by Baldwin and Sivek \cite{baldwin2018khovanov} to relate different sutures on the knot complements. The first author further studied a graded version of bypass exact triangle in \cite{li2019direct}.

\begin{thm}\label{thm: graded bypass}
Suppose $K\subset S^3$ is a knot. For any $i\in\intg$, there are three exact triangles
\begin{equation}\label{eq: graded exact triangle, +}
    \xymatrix{
    SHI(S^3(K), \gamma_{(1,-2g)},i+\frac{1}{2})\ar[r]^{\psi_{+,i}}&SHI(S^3(K), \gamma_{(1,-2g-1)}, i)\ar[dl]\\
    KHI(K, i-g)\ar[u]
    }
\end{equation}

\begin{equation}\label{eq: graded exact triangle,-}
    \xymatrix{
    SHI(S^3(K), \gamma_{(1,-2g)},i-\frac{1}{2})\ar[r]^{\psi_{-,i}}&SHI(S^3(K), \gamma_{(1,-2g-1)}, i)\ar[dl]\\
    KHI(K, i+g)\ar[u]
    }
\end{equation}

\begin{equation}\label{eq: graded exact triangle, extended}
    \xymatrix{
    SHI(S^3(K), \gamma_{(1,-2g)},g+i+\frac{1}{2})\ar[r]^{}&SHI(S^3(K), \gamma_{(2,-4g-1)}, i)\ar[dl]\\
    SHI(S^3(K), \gamma_{(1,-2g-1)}, -g+i)\ar[u]
    }
\end{equation}
\end{thm}

One of the main result of \cite{baldwin2018khovanov} can be re-stated as follows.
\bthm\label{thm: vanishing psi}
Suppose $K$ is a fibred knot and is not right veering. Then the map
$$\psi_{+,2g-1}:SHI(S^3(K), \gamma_{(1,-2g)},2g-\frac{1}{2})\ra SHI(S^3(K), \gamma_{(1,-2g-1)}, 2g-1)$$
in (\ref{eq: graded exact triangle, +}) is zero.
\ethm

Next, we introduce the Heegaard diagrams of $3$-manifolds and knots.

\bdefn
A \textbf{(genus $g$) diagram} is a triple $(\Sigma,\al,\be)$ so that the followings hold.
\begin{enumerate}
    \item We have $\Sigma$ being a connected closed surface of genus $g$.
    \item We have $\al=\{\al_1,\dots,\al_m\}$ and $\be=(\be_1,\dots,\be_n)$ being two sets of pair-wise disjoint simple closed curves on $\Sigma$. We do not distinguish the set and the union of curves.
\end{enumerate}

A \textbf{(genus $g$) Heegaard diagram} is a (genus $g$) diagram $(\Sigma,\al,\be)$ satisfying the following conditions.
\begin{enumerate}
    \item We have $|\al|=|\be|=g$, \textit{i.e.}, there are $g$ many curves in either tuple.
    \item The complements $\Sigma\backslash \al$ and $\Sigma\backslash\be$ are connected.
\end{enumerate}
\edefn

It is a basic fact in low dimensional topology.
\bthm
Any knot $K\subset S^3$ admits a Heegaard diagram.
\ethm

\section{Knots diagrams and instanton knot homology}
In this section, we prove Theorem \ref{thm: algorithm}.
\begin{proof}[Proof of theorem \ref{thm: algorithm}]
Suppose $K\subset S^3$ is a knot and $D$ is a knot diagram of the knot.

{\bf Step 1}. We construct a Heegaard diagram from a knot diagram. To do this, we first form the singular knot $K_s$ by replacing every crossing of $D$ with two arcs that intersect at one point. We can think of $K_s$ as embedded in $S^3$. Then let $H=S^3\backslash N(K_s)$. It is straightforward to check that $H$ is a handle body. Let $\Sigma=\partial H$ and the $\alpha$-curves consists of $g=c(D)+1$ meridians of $H$, where $g$ is the genus of $H$ and $c(D)$ denotes the number of crossings in $D$. Next, we need to find the $\beta$-curves. We need $g$ many of them. We draw one $\beta$-curve around each crossing of $K$ according to the principle shown in Figure \ref{curve around crossing} (note there are $c(K)=g-1$ many) and pick a meridian of $K$ to be the last $\beta$-curve. As in \cite{hom2020notes}, this gives us a Heegaard diagram of $K$.

\begin{figure}[ht]
\centering
\begin{overpic}[width=5.5in]{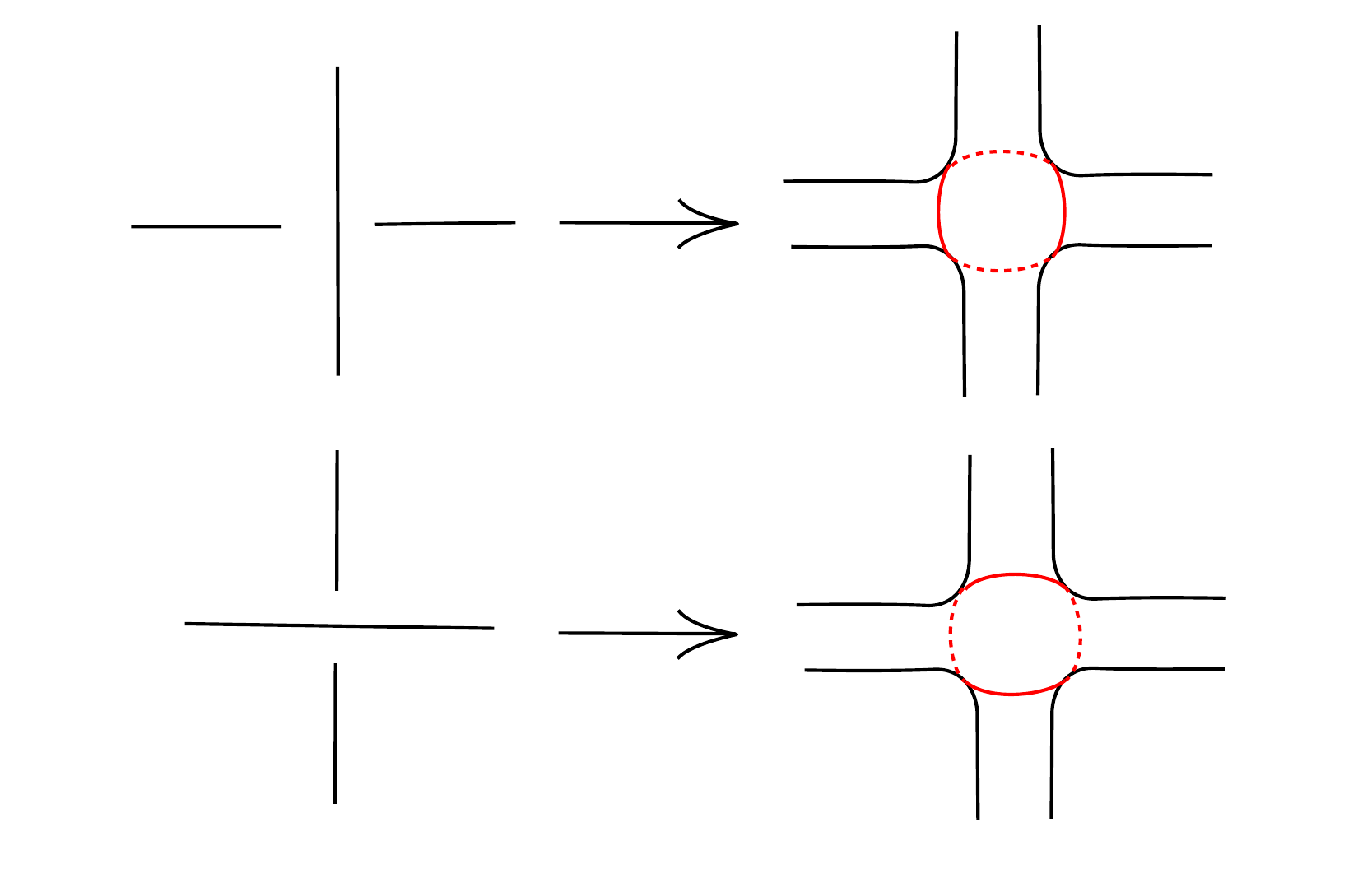}
    
\end{overpic}
\vspace{0.05in}
\caption{The red curves on the right are the $\beta$-curves.}\label{curve around crossing}
\end{figure}

{\bf Step 2}. We construct a sutured handle body from $(\Sigma,\al,\be)$. We simply pick $H=S^3\backslash N(K_s)$ to be the handle body, and all $\be$-curves are the components of $\ga$. Also, $\ga$ has one last component obtained by $(g-1)$ band sums on $\be$-curves that make all $g$ many $\beta$-curves into one connected simple closed curve that can be isotoped to be disjoint from all of the original $\beta$-curves. Theorem \ref{thm: key inequality} then applies and we have
$${\rm dim}_{\mathbb{C}}KHI(K)\leq {\rm dim}_{\mathbb{C}}SHI(H,\ga).$$

{\bf Step 3}. We compute an upper bound of $\dim_{\mathbb{C}}SHI(H,\ga)$. This is done by induction based on the following two lemmas.

To present the first lemma, recall we have $g$ many $\alpha$-curves. Call them $\al_1$,..., $\al_g$.

\blem[Kronheimer and Mrowka \cite{kronheimer2010knots}]\label{lem: rank at most 1} If for every index $i\in\{1,...,g\}$, we have
$$|\al_i\cap \ga|\leq 2,$$
where $|\cdot|$ denotes the number of intersection points, then 
$$dim_{\mathbb{C}}SHI(H,\ga)\leq 1.$$
\elem

\blem[Baldwin and Sivek \cite{baldwin2018khovanov}]\label{lem: bypass triangle} Suppose $i\in\{1,...,g\}$ and $\be\subset\partial D_i$ is part of $\partial D_i$ so that
$$\partial\be\subset \ga~{\rm and~}|\be\cap\ga|=3.$$
See Figure \ref{fig_by_pass} for an example. Within a neighborhood of $\be$, we can alter the suture $\ga$ as shown in Figure \ref{fig_by_pass}, and get two new sutures $\ga'$ and $\ga''$. Then we have
$$dim_{\mathbb{C}}SHI(H,\ga)\leq dim_{\mathbb{C}}SHI(H,\ga')+dim_{\mathbb{C}}SHI(H,\ga'').$$
\elem

It is clear that
$$|\ga'\cap D_i|\leq |\ga\cap D_i|-2~{\rm and~}|\ga''\cap D_i|\leq |\ga\cap D_i|-2.$$
So using Lemma \ref{lem: rank at most 1} we can reduce the number of intersections of $\ga$ with arbitrary meridian disk of $H$. When $\ga$ intersects all meridian disks at most two times, Lemma \ref{lem: bypass triangle} applies and we can obtain a bound on $dim_{\mathbb{C}}SHI(H,\ga)$ for any suture $\ga$.

	

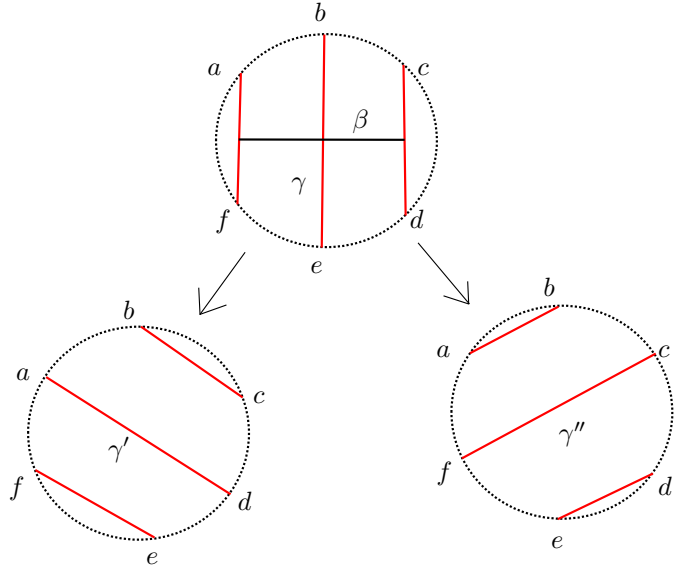
\begin{figure}[h]
\centering
\begin{overpic}[width=4.0in]{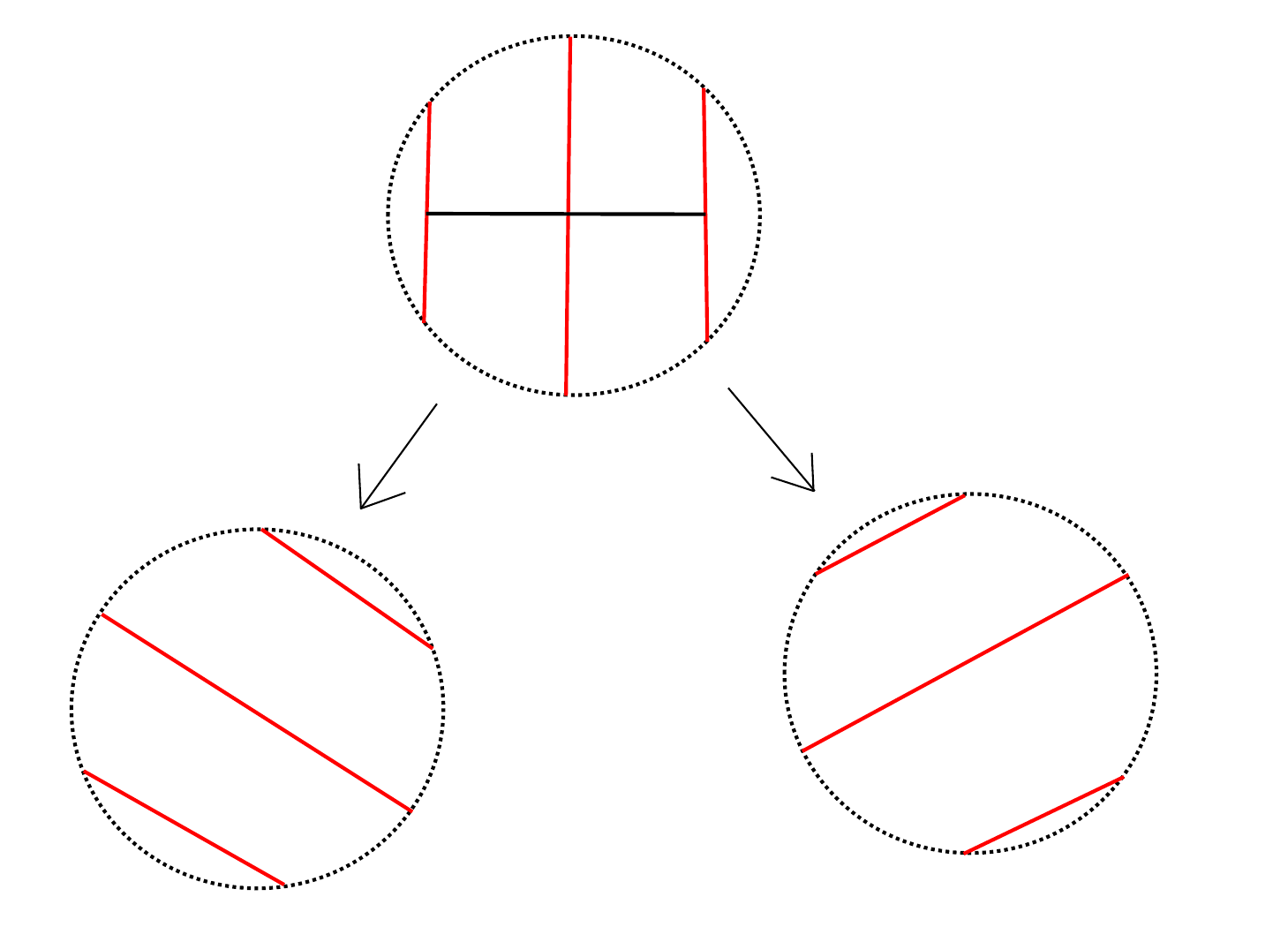}
	\put(4,25){$a$}
	\put(18,33){$b$}
	\put(35,22){$c$}
	\put(33,8){$d$}
	\put(21,1){$e$}
	\put(3,10){$f$}
	\put(16,15){$\ga'$}
	
	\put(59,28){$a$}
	\put(73,36){$b$}
	\put(88,28){$c$}
	\put(88,10){$d$}
	\put(74,3){$e$}
	\put(59,12){$f$}
	\put(75,17){$\ga''$}
	
	\put(29,65){$a$}
	\put(43,72){$b$}
	\put(56.5,65){$c$}
	\put(55.5,45){$d$}
	\put(42.5,39){$e$}
	\put(30,45){$f$}
	\put(40,50){$\ga$}
	\put(48,58){$\be$}
\end{overpic}
\vspace{0.1in}
\caption{A by-pass move obtaining $\ga_1$ and $\ga_2$ from $\ga$. The red curves are the sutures. The dotted circle bounds the disk $E\subset \Sigma_n$.}\label{fig_by_pass}
\end{figure}

\end{proof}

We have performed computations for all knots with crossing number at most $7$, as well as a more complicated knot $10_{153}$. The results are summarized in Table \ref{tab: knots}. 
\begin{table}[h!]
\begin{center}
\begin{tabular}{ |c|c|c| } 
 \hline
 Knots & Upper bound for $\dim_{\mathbb{C}}KHI$ & Alexander polynomial \\ 
 \hline
 $3_1$ & 3 & $t-1+t^{-1}$ \\ 
 \hline
 $4_1$ & 5 & $-t+3-t^{-1}$ \\ 
  \hline
 $5_1$ & 5 & $t^2-t+1-t^{-1}+t^{-2}$ \\ 
  \hline
 $5_2$ & 7 & $2t-3+2t^{-1}$ \\ 
  \hline
 $6_1$ & 9 & $-2t+5-2t^{-1}$\\ 
  \hline
 $6_2$ & 11 & $-t^2+3t-3+3t^{-1}-t^{-2}$ \\ 
  \hline
 $6_3$ & 13 & $t^2-3t+5 -3 t^{-1} + t^{-2}$ \\ 
  \hline
 $7_1$ & 7 & $t^3 -t^2+t-1+t^{-1}-t^{-2}+t^{-3}$ \\ 
  \hline
 $7_2$ & 11 & $3t+5 + 3t^{-1}$ \\ 
  \hline
 $7_3$ & 13 & $2t^2-3t+3-3t^{-1}+2t^{-2}$ \\ 
  \hline
 $7_4$ & 15 &  $4t-7+4t^{-1}$\\ 
 \hline
 $7_5$ & 17 &  $2t^2-4t+5-4t^{-1} + 2t^{-2}$\\ 
  \hline
 $7_6$ & 19 & $-t^2+5t-7+5t^{-1}-t^{-2}$ \\ 
  \hline
 $7_7$ & 21 & $t^2-5t+9-5t^{-1}+t^{-2}$ \\  
 \hline
 $10_{153}$ & 17 & $t^3-t^2-t^1+3-t^{-1}-t^{-2}+t^{-3}$\\
 \hline
\end{tabular}
\end{center}
\caption{Knots with small crossings. We use Rolfsen's knot table in \cite{rolfsen2003knot} to name all these knots. To obtain a lower bound (which coincide with the upper bound), one can sum up the abstract value of all coefficients of the Alexander polynomial.}\label{tab: knots}
\end{table}

\brem
Besides knots with small crossings, we also tried another knot $10_{153}$. The reason why we work on this particular knot is the following. As explained in the introduction, for alternating knots, the upper bounds of the dimension of $KHI$ coincide with the lower bound coming from the Alexander polynomial. In \cite{li2020heegaard}, the first author and his collaborator computed upper bounds for all $(1,1)$-knots, and for many families of $(1,1)$-knots, the upper bounds obtained in \cite{li2020heegaard} are better than those from Khovanov homology. Hence we are interested in finding more examples outside the range of alternating knots and $(1,1)$-knots. We didn't find a complete list for all $(1,1)$-knots, so we turn to search in the knots with tunnel number at least $2$, since all $(1,1)$-knots are known to have tunnel number $1$. So $10_{153}$ is the first knot $K$ came into our sight that satisfies the following three conditions:

\begin{enumerate}
\item The knot has tunnel number at least $2$ and is not alternating.
\item The upper bound from Khovanov homology is strictly larger than the lower bound from the Alexander polynomial.
\item The Alexander polynomial of the knot is not too complicated.
\end{enumerate}

Unfortunately, the upper bound we obtained for $10_{153}$, which is $17$, coincides with the upper bound from Khovanov homology. Note this upper bound is strictly greater than the lower bound from Alexander polynomial, so the precise dimension of $KHI$ for $10_{153}$ is still open. 
\erem

\section{Dehn surgeries on knots}
In this section, we prove Theorem \ref{thm : main theorem}.

\begin{proof}[Proof of Theorem \ref{thm : main theorem}]
Suppose $K\subset S^3$ is a knot of genus $g$ and its instanton knot homology satisfies the assumption in the hypothesis of the theorem. Note from the assumption that $KHI(K,g)\cong\mathbb{C}$ and Theorem \ref{thm: properties of KHI}, we know that $K$ is fibred and $g\geq 2$. Then either $K$ or the mirror of $K$ is not right veering. By passing to its mirror if necessary, we can assume that $K$ itself is not right veering.  We begin with a few lemmas.
\begin{lem}\label{lem : first top grading}
We have	
\begin{equation}\label{eq: top for 2g+1}
SHI(S^3(K), \gamma_{(1,-2g-1)}, 2g) \cong\mathbb{C}
\end{equation}
and
\begin{equation}\label{eq: top for 2g}
SHI(S^3(K), \gamma_{(1,-2g)}, 2g-\frac{1}{2}) \cong\mathbb{C}.
\end{equation}
\end{lem}

\begin{proof}
Applying term 2 of Theorem \ref{thm: properties of KHI}, Formula (\ref{eq: graded exact triangle, +}) from Theorem \ref{thm: graded bypass}, and the assumption that $KHI(K,g)\cong\mathbb{C}$ in the hypothesis, we conclude (\ref{eq: top for 2g+1}). Similarly, (\ref{eq: top for 2g}) follows from term 2 of Theorem \ref{thm: properties of KHI}, Formula (\ref{eq: top for 2g+1}), and Formula (\ref{eq: graded exact triangle, +}) from Theorem \ref{thm: graded bypass}.
\end{proof}

\begin{lem}\label{lem : second top grading}

 $SHI(S^3(K), \gamma_{(1,-2g-1)}, 2g-1) =0$.	
\end{lem}

\begin{proof}
	Note we have argued that $K$ is fibred and also assumed that it is not right veering. Hence Theorem \ref{thm: vanishing psi} applies. Then the lemma follows from Theorem \ref{thm: vanishing psi}, Formula (\ref{eq: top for 2g}), the assumption that $KHI(K,g-1)\cong\mathbb{C}$, and Theorem \ref{thm: graded bypass}.
\end{proof}

\begin{lem}\label{lem: identifying gradings}
We have
$$SHI(S^3(K), \gamma_{(1,-2g-1)},  i) \cong SHI(S^3(K), \gamma_{(1,-2g)}, i-\frac{1}{2})$$ for $i>1$.
\end{lem}

\begin{proof}
Since $-i-g< -g$, we have $KHI(K, -i-g) = 0 $ by term 2 of Theorem \ref{thm: properties of KHI}.
Thus, by Theorem \ref{thm: graded bypass}, we have more isomorphisms
\begin{equation}\label{eq: induction 1}
	 SHI(S^3(K), \gamma_{(1,-2g-1)},  -i) \cong SHI(S^3(K), \gamma_{(1,-2g)},  -i + \frac{1}{2}).
\end{equation}

By term 3 of Theorem \ref{thm: properties of SHI}, we have an isomorphism
\begin{equation}
	 SHI(S^3(K), \gamma_{(1,-2g-1)},  -i) \cong SHI(S^3(K), \gamma_{(1,-2g-1)},  i)
\end{equation}

and 

\begin{equation}
	 SHI(S^3(K), \gamma_{(1,-2g)},  -i + \frac{1}{2})  \cong SHI(S^3(K), \gamma_{(1,-2g)},  i- \frac{1}{2}) 
\end{equation}

The lemma then follows after substituting in Equation (\ref{eq: induction 1}).

\end{proof}

\begin{lem} \label{lem : complete gradings}
We have
\[
    SHI(S^3(K), \gamma_{(1,-2g-1)},i)=\left\{
                \begin{array}{ll}
                  0 &  g+1 \leq i \leq 2g-1 \\ 
                   \mathbb{C} &  2 \leq i \leq g\\
                \end{array}
              \right.
  \]
\end{lem}

 \begin{proof}
To start,  in Lemma \ref{lem : second top grading}, we have proved that 
$$SHI(S^3(K), \gamma_{(1,-2g-1)}, 2g-1)  \cong 0.$$
By Lemma \ref{lem: identifying gradings}, we know that
$$SHI(S^3(K),\gamma_{(1,-2g)},2g-1-\frac{1}{2})\cong SHI(S^3(K), \gamma_{(1,-2g-1)}, 2g-1)=0.$$
Since $KHI(K,g-2)=0$ by the hypothesis of the theorem, taking $i=2g-2$ in (\ref{eq: graded exact triangle, +}), we have
\begin{equation*}
	SHI(S^3(K), \gamma_{(1,-2g-1)}, 2g-2)\cong SHI(S^3(K),\gamma_{1,-2g},2g-1-\frac{1}{2})=0.
\end{equation*}

Repeating the above argument once more for $i=2g-3$, we conclude that 
$$SHI(S^3(K), \gamma_{(1,-2g-1)}, 2g-3)\cong 0.$$
We can keep running this argument until we finish the case $i=g+1$, where we have
\begin{equation} \label{eq: g+1}
	SHI(S^3(K), \gamma_{(0,-2g-1)}, g+1)=0.
\end{equation}

By far we have proved the vanishing part of the lemma. For the other half of the lemma, we use a similar argument. Note we have
$$KHI(K, 0) \cong \mathbb{C}$$
by the hypothesis of the theorem, and
$$SHI(S^3(K), \gamma_{(1,-2g)},g + \frac{1}{2}) \cong 0$$
by Lemma \ref{lem: identifying gradings} and equation (\ref{eq: g+1}). Taking $i=g$ in Formula (\ref{eq: graded exact triangle, +}), we conclude $SHI(S^3(K), \gamma_{(1,-2g-1)}, g) \cong \mathbb{C}$. Then using a similar repetitive argument as above, we conclude
\[SHI(S^3(K), \gamma_{(1,-2g-1)},i)\cong \mathbb{C} \]
for $2 \leq  i \leq g$.
\end{proof}

\begin{lem}\label{lem: middle 4g+1}
For any $i\in\intg$ such that $-g\leq i\leq g$, we have


$$SHI(S^3(K), \gamma_{(2, -4g-1)}, i) \cong \mathbb{C} $$
	
\end{lem}

\begin{proof}


First, taking $i = g$, by term 1 of Theorem \ref{thm: properties of SHI}, we know that 

\[SHI(S^3(K), \gamma_{(1,-2g)}, g+i +\frac{1}{2}) = SHI(S^3(K), \gamma_{(1,-2g)}, 2g +\frac{1}{2}) = 0.\]

Also, by term 4 of Theorem \ref{thm: properties of SHI}, we have
\[SHI(S^3(K), \gamma_{(1,-2g-1)}, -g +i) = SHI(S^3(K), \gamma_{(1,-2g-1)}, 0) \cong  \mathbb{C}.\]

Then by Formula (\ref{eq: graded exact triangle, extended}) in Theorem \ref{thm: graded bypass}, we have
\[ SHI(S^3(K), \gamma_{(2, -4g-1)}, g) \cong  SHI(S^3(K), \gamma_{(2, -4g-1)}, -g)\cong \mathbb{C}.\]


Next, we skip the case $i=g-1$ and consider $i$ so that $ g-2 \geq i \geq -g+1.$ The case $i=g-1$ will be dealt with later. 

By Lemma \ref{lem : complete gradings}, when $i\in[0,g-2]$, we have
\[ SHI(S^3(K), \gamma_{(1,-2g-1)}, g+i +1) = 0.\]

 Also, by Lemma \ref{lem: identifying gradings}, we have
 \[SHI(S^3(K), \gamma_{(1,-2g)}, g+i +\frac{1}{2}) \cong SHI(S^3(K), \gamma_{(1,-2g-1)}, g+i +1).\]

Thus, we derive
$$SHI(S^3(K), \gamma_{(1,-2g)}, g+i +\frac{1}{2}) = 0$$
for all $i\in[0,g-2]$.
 
On the other hand, by term 3 of Theorem \ref{thm: properties of SHI}, we have an isomorphism 
 \[SHI(S^3(K), \gamma_{(1,-2g-1)}, i-g) \cong SHI(S^3(K), \gamma_{(1,-2g-1)}, g-i),\]
 and by Lemma \ref{lem : complete gradings}, for any $i\in[0,g-2]$, we have
 \[ SHI(S^3(K), \gamma_{(1,-2g-1)}, g-i) \cong \mathbb{C}.\]
Thus, by Formula (\ref{eq: graded exact triangle, extended}), we have 
 
  \[SHI(S^3(K), \gamma_{(2, -4g-1)}, i) \cong \mathbb{C}.\]

  
  
Similarly, for $-g+1\leq i \leq -1$,  by Lemma \ref{lem : complete gradings}, we have
  \[ SHI(S^3(K), \gamma_{(1,-2g-1)}, g-i) \cong 0,\]
and 
  \[ SHI(S^3(K), \gamma_{(1,-2g)}, g+i + \frac{1}{2}) \cong  SHI(S^3(K), \gamma_{(1,-2g-1)}, g+i + 1)\cong \mathbb{C}.\]
Then, by Formula (\ref{eq: graded exact triangle, extended}), we conclude 
\[SHI(S^3(K), \gamma_{(2, -4g-1)}, i) \cong \mathbb{C}\]
 for $g-2 \geq i \geq -g+1$.
 
 Notice by far we have covered all $-g \leq i \leq g$ except $i = g-1$, which could be reached by term 3 in Theorem \ref{thm: properties of SHI} and the conclusion above since
\[ SHI(S^3(K), \gamma_{(2, -4g-1)}, g-1) \cong SHI(S^3(K), \gamma_{(2, -4g-1)}, 1-g) \cong \mathbb{C} \]
The lemma then follows.
\end{proof}

Theorem \ref{thm : main theorem} then follows directly from Lemma \ref{lem: middle 4g+1} and term 5 of Theorem \ref{thm: properties of SHI}. 
\end{proof}

\bibliographystyle{alpha}
\bibliography{index}

\end{document}